\documentclass[11pt, reqno]{amsart}
\usepackage[all,tips]{xy}
\usepackage{latexsym, amsfonts, amsmath, amssymb, mathrsfs, graphicx, xcolor, ytableau, hyperref, float}
\usepackage[justification=centering]{caption}
\allowdisplaybreaks
\usepackage{centernot}
\usepackage{tikz}
\usepackage{pgfplots}

\definecolor{red}{rgb}{1,0,0}
\definecolor{magenta}{rgb}{1,0,1}
\definecolor{dartmouthgreen}{rgb}{0.05, 0.5, 0.06}
\definecolor{purple(x11)}{rgb}{0.63,0.36,0.94}
\definecolor{turquoise}{rgb}{0.25, 0.87, 0.81}
\newtheorem{theorem}{Theorem}[section]
\newtheorem{lemma}[theorem]{Lemma}
\newtheorem{proposition}[theorem]{Proposition}
\newtheorem{corollary}[theorem]{Corollary}
\newtheorem{conjecture}[theorem]{Conjecture}

\newtheorem*{question}{Question}

\theoremstyle{definition}
\newtheorem{remark}[theorem]{Remark}

\newcommand{\Cal}[1]{\ensuremath{\mathcal{#1}}}

\newcommand{\lp}{\left(}
\newcommand{\rp}{\right)}

\newcommand{\N}{\mathbb N}
\usepackage[textheight=8.75in, textwidth=6.75in]{geometry}

\def\C{{\mathbb C}}

\def\Z{{\mathbb Z}}

\def\Q{{\mathbb Q}}

\def\O_K{{\Cal{O}_{K}}}
\def\O_F{{\Cal{O}_{F}}}
\def\N_F{{\Cal{N}_{F/\Q}}}

\newcommand{\ord}{\mathrm{ord}}

\newcommand{\legendre}[2]{\genfrac{(}{)}{}{}{#1}{#2}}

\def\O_K{{\Cal{O}_{K}}}
\def\O_F{{\Cal{O}_{F}}}
\def\N_F{{\Cal{N}_{F/\Q}}}

\numberwithin{equation}{section}
\numberwithin{theorem}{section}

\title{Universal Ramanujan-type congruences for prime-detecting quasimodular forms}

\author{Bailey Chrobak}
\address{ Sussex Central High School, 26026 Patriots Way, Georgetown, DE 19947}
\email{baileychrobak@icloud.com}

\author{William Craig}
\address{Department of Mathematics, United States Naval Academy, 572C Holloway Road
Mail Stop 9E. Annapolis, MD 21402}
\email{wcraig@usna.edu}

\begin{document}

\maketitle

\begin{abstract}
    We establish the existence of infinitely many Ramanujan-type congruences which hold uniformly for the full $\Z$-module of prime detecting quasimodular forms. In particular, we prove that for every prime $p$, there is at least one arithmetic progression $An+B$ such that every prime-detecting quasimodular form has coefficients divisible by $p$ in this progression.
\end{abstract}

\section{Introduction}

Let $p(n)$ denote the number of partitions of $n$; that is, the number of distinct ways to represent $n$ as a sum of positive integers, treating two representations as identical if they differ only in permuting the summands. The partition function, and its many variations, are some of the most important functions studied in number theory and combinatorics. Among the most famous facts about partitions are Ramanujan's congruence relations \cite{Ramanujan}
\begin{align*}
    p\lp 5n+4 \rp &\equiv 0 \pmod{5}, \\
    p\lp 7n+5 \rp &\equiv 0 \pmod{7}, \\
    p\lp 11n+6 \rp &\equiv 0 \pmod{11}.
\end{align*}
The tools employed by Ramanujan and others after him in proving these conjectures motivated the study of similar congruences for families of {\it modular forms} and {\it quasimodular forms}. For the purposes of this paper, we will be concerned only with positive weight quasimodular forms for the full congruence subgroup $\mathrm{SL}_2(\Z)$, and we define such forms as elements of the algebra
\begin{align*}
    \mathcal M := \C[G_2, G_4, G_6, G_8, G_{10}, \dots],
\end{align*}
where $G_k$ are the Eisenstein series of weight $k$, defined for $k \geq 2$ even by
\begin{align*}
    G_k(z) := -\dfrac{B_k}{2k} + \sum_{n \geq 1} \sigma_{k-1}(n) q^n,
\end{align*}
where $q := e^{2\pi i z}$, $B_k$ are the Bernoulli numbers, and $\sigma_{k-1}(n) := \sum_{d|n} d^{k-1}$. We will use the notation $\mathcal M^\Z$ to denote the submodule $\mathcal M \cap \lp \Q \oplus q\Z[[q]] \rp$ of $\mathcal M$, i.e. of those quasimodular forms whose non-constant coefficients are integers and whose constant term is rational. We note that every $G_k$ belongs to $\mathcal M^\Z$, and if we define the differential operator
\begin{align*}
    D := q \dfrac{d}{dq},
\end{align*}
both $\mathcal M$ and $\mathcal M^\Z$ are closed under the action of $D$. For an overview of the basic theory of quasimodular forms, see for instance \cite{KanekoZagier}.

We will be concerned with congruences like Ramanujan's for a certain $\Z$-submodule $\Omega$ of $\mathcal M^\Z$ defined by the first author with van Ittersum and Ono \cite{CraigIttersumOno} as the space of {\it prime-detecting quasimodular forms}. We call an element $f = \sum_{n \geq 0} a_f(n) q^n \in \mathcal M^\Z$ a {\it proper prime-detecting quasimodular form}\footnote{In \cite{CraigIttersumOno}, this was the definition of a prime-detecting form.} if, for $n \geq 2$, we have $a_f(n) \geq 0$ with equality if and only if $n$ is prime, and we call $f$ a {\it prime-detecting quasimodular form} if $f$ belongs to the $\Z$-module $\Omega$ generated by all proper prime-detecting forms. Combining the main results of \cite{CraigIttersumOno} and of Kane, Krishnamoorthy and Lau \cite{KaneKrishnamoorthyLau}, the elements of the $\Z$-module $\Omega$ are now classified; every $f \in \Omega$ is a finite $\Z$-linear combination of certain proper prime-detecting forms $H_k$
for even $k \geq 6$ along with their derivatives $D^m H_k$ for $m \geq 1$. The forms $H_k$ are defined by
\begin{align*}
    H_k := \sum_{n \geq 0} b_k(n) q^n,
\end{align*}
where
\begin{align*}
    b_k(n) := \begin{cases}
        \dfrac{(n^2-n+1) \sigma_1(n) - \sigma_3(n)}{6} & \textrm{for }k = 6, \\
        \dfrac{-n^2 \sigma_{k-7}(n) + (n^2+1) \sigma_{k-5}(n) - \sigma_{k-3}(n)}{24} & \textrm{for } k \geq 8.
    \end{cases}
\end{align*}
It is shown in \cite[Lemma 12]{CraigIttersumOno} that each $D^m H_k$ has integer coefficients with no common divisor, i.e. there is no prime $p$ for which $p|b_k(n)$ for every $n$.

When studying the existence of Ramanujan congruences for the quasimodular forms $H_k$, a striking pattern emerges. Not only can we identify many such congruences for individual forms $H_k$, even some with comparatively large primes, but there are many examples of congruences that hold {\it uniformly} for all $H_k$ at once. We refer to these congruences as {\it universal Ramanujan-type congruences} for $\Omega$. Since the functions $b_k(n)$ are constructed from divisor sums and multiples of $n$, these congruences are closely linked to the framework of Ramanujan-type congruences for multiplicative functions studied by the second author and Merca \cite{CraigMerca}.

Given the infinitude of $H_k$ available, it would be natural to ask whether there can be an infinite number of such congruences. We can answer this question in the affirmative using Serre's theory of $\ell$-adic Galois representations associated to modular forms, in a way analogous to the work of Ono \cite{Ono} on congruences for $p(n)$.

\begin{theorem} \label{Thm: Infinitely many uniform congruences}
    Let $M \geq 2$ be an integer. Then there are infinitely many non-nested arithmetic progressions $An+B$ such that 
    \begin{align*}
        b_k(An+B) \equiv 0 \pmod{M}
    \end{align*}
    for all even $k \geq 6$ and for all $n \geq 0$.
\end{theorem}

This method of proof is, however, far from explicit. It is natural, therefore, to pursue explicit results which instantiate Theorem \ref{Thm: Infinitely many uniform congruences} by other methods. The main body of our paper is dedicated precisely to this task. The culmination of our explicit results on Theorem \ref{Thm: Infinitely many uniform congruences} may be summarized in the following, which is an immediate consequence of Theorems \ref{Thm: mod 3 congruences}, \ref{Thm: even progressions}, and \ref{Thm: Multiplicative Congruences}.

\begin{theorem} \label{Thm: mult congruences corollary}
    Let $p$ be a prime. Then we may construct explicit infinite families of non-nested arithmetic progressions $An+B$ such that 
    \begin{align*}
        b_k(An+B) \equiv 0 \pmod{p}
    \end{align*}
    for all even $k \geq 6$ and for all $n \geq 0$.
\end{theorem}

\begin{remark}
    Since the forms $H_k$ and their derivatives are a $\Z$-basis of the $\Z$-module $\Omega$, it follows that any universal congruence for the forms $H_k$ is inherited by every $f \in \Omega$. As we discuss in Section \ref{Sec: Conclusion}, many $f \in \Omega$ have congruences which are not universal.
\end{remark}

We make note of a particularly special case of Theorem \ref{Thm: mult congruences corollary} in which the values of $A,B$ are more closely tied to $p$ than in the typical cases. To state this corollary, we recall that a {\it Mersenne prime} is a prime of the form $2^n-1$.

\begin{corollary} \label{Cor: Mersenne primes}
    Let $p \not = 7$ be a Mersenne prime. Then we have
    \begin{align*}
        b_k\lp (p+1)^pn + (p+1)^{p-1} \rp \equiv 0 \pmod{p}
    \end{align*}
    for all even $k \geq 6$ and for all $n \geq 0$.
\end{corollary}

The remainder of our paper is summarized as follows. In Section \ref{Sec: Isolated}, we prove a number of isolated universal congruences for small odd primes. In Section \ref{Sec: Quad Nonresidues}, we prove two infinite families of universal Ramanujan-type congruences modulo 2 by utilizing the elementary theory of quadratic residues modulo odd primes. In Section \ref{Sec: Serre}, we prove Theorem \ref{Thm: Infinitely many uniform congruences} as a corollary of classical results of Serre. In Section \ref{Sec: Multiplicative}, we prove Theorem \ref{Thm: mult congruences corollary}. Finally, in Section \ref{Sec: Conclusion} we discuss possible future directions for work.

\section*{Acknowledgments}

This project was conducted during the course of the Science and Engineering Apprenticeship Program (SEAP) at the US Naval Academy during the summer of 2026. The first author thanks SEAP for funding which enabled this visit to occur. The views expressed in this article are those of the author and do not reflect the official policy or position of the U.S. Naval Academy, Department of the Navy, the Department of War, or the U.S. Government.

\section{Isolated congruences} \label{Sec: Isolated}

In this section, we discuss a number of universal Ramanujan-type congruences for a few small odd primes.

\subsection{Congruences modulo 3}

\begin{theorem} \label{Thm: mod 3 congruences}
    For every $n \geq 0$ and $k \geq 6$, we have
    \begin{align*}
        b_k(3n+1) \equiv b_k(3n+2) \equiv 0 \pmod{3}.
    \end{align*}
\end{theorem}

\begin{proof}
    First, we consider the case $k=6$. By the definition of $b_6(n)$, the desired result follows from
    \begin{align*}
        (n^2-n+1) \sigma_1(n) - \sigma_3(n) \equiv 0 \pmod{9}
    \end{align*}
    for $n \equiv 1, 2 \pmod{3}$. If we replace $n \mapsto 9n+a$ for some choice $a \equiv 1,2,4,5,7,8 \pmod{9}$, the desired result follows whenever
    \begin{align*}
        (a^2-a+1)\sigma_1(9n+a) - \sigma_3(9n+a) \equiv 0 \pmod{9}.
    \end{align*}
    If $a \equiv 2 \pmod{3}$, then we may calculate $a^2 - a + 1 \equiv 3 \pmod{9}$, and therefore in this case we must check that
    \begin{align*}
        3\sigma_1(3n+2) - \sigma_3(3n+2) \equiv - \sum_{d|3n+2} \lp d^3 - 3d \rp \equiv 0 \pmod{9}.
    \end{align*}
    Now, $3n+2$ is never a square, so in the above summation over $d|n$ we may pair the terms associated to $d \pmod{9}$ and $d^{-1} \pmod{9}$. If we write $d' = \frac{3n+2}{d}$ and $d^{-1} d \equiv 1 \pmod{9}$, then $(d')^3 \equiv -d^{-3} \pmod{9}$ and $3d' \equiv 6d^{-1} \pmod{9}$. Since Euler's theorem dictates $d^3 \equiv d^{-3} \pmod{9}$ and $d^2 \equiv 1 \pmod{3}$ for $d|3n+2$, we have
    \begin{align*}
        \sum_{d|n} \lp d^3 - 3d \rp &\equiv \sum_{d,d'|3n+2} \lp d^3 - 3d + (d')^3 - 3d' \rp \pmod{9} \\
        &\equiv \sum_{d,d' | 3n+2} \lp d^3 - 3d - d^{-3} - 6d^{-1} \rp \pmod{9} \\
        &\equiv - \sum_{d,d' | 3n+2} 3\lp d + 2d^{-1} \rp \equiv 0 \pmod{9},
    \end{align*}
    since $d + 2d^{-1} \equiv d + 2d \equiv 0 \pmod{3}$. If on the other hand $a \equiv 1 \pmod{3}$, then we must show
    \begin{align*}
        (3n+1) \sigma_1(3n+1) - \sigma_3(3n+1) \equiv 0 \pmod{9}.
    \end{align*}
    Writing $3n+1 \equiv dd^{-1} \pmod{9}$, we obtain
    \begin{align*}
        (3n+1) \sigma_1(3n+1) - \sigma_3(3n+1)
        \equiv \sum_{d|n} (dd^{-1}) d - d^3
        \equiv \sum_{d|n} d^2\lp d^{-1} - d \rp \pmod{9}.
    \end{align*}
    When $d \equiv 1, 8 \pmod{9}$, we have $d^{-1} - d \equiv 0 \pmod{9}$. Otherwise, $\{ d, d^{-1} \} \equiv \{ 2, 5 \} \textrm{ or } \{ 4, 7 \} \pmod{9}$; in particular, we note that $d \not\equiv d^{-1} \pmod{9}$ and so we may pair terms in the sum over $d|n$ and we obtain in either case that
    \begin{align*}
        d^2\lp d^{-1} - d \rp + d^{-2}\lp d - d^{-1} \rp \equiv \lp d^2 - d^{-2} \rp \lp d - d^{-1} \rp \equiv 0 \pmod{9}.
    \end{align*}
    This completes the proof for $k=6$.

    For $k \geq 8$, to prove the desired result we must show that
    \begin{align*}
        -n^2 \sigma_{k-7}(n) + (n^2+1) \sigma_{k-5}(n) - \sigma_{k-3}(n) \equiv 0 \pmod{9}
    \end{align*}
    for $n \equiv 1,2 \pmod{3}$. If we write $3n+a = d d'$ for each $d|3n+a$ and $a \equiv 1, 2 \pmod{3}$, the desired result follows if
    \begin{align*}
        \sum_{d|n} \left[ -d^2 (d')^2 d^{k-7} + (d^2 (d')^2 + 1)d^{k-5} - d^{k-3} \right] \equiv 0 \pmod{9}.
    \end{align*}
    Among those $d|3n+a$ for which $d^2 \equiv 1 \pmod{9}$, i.e. $d \equiv 1, 8 \pmod{9}$, we also have $(d')^2 \equiv 1 \pmod{9}$ and the summand above vanishes since
    \begin{align*}
        -d^2 (d')^2 d^{k-7} + (d^2 (d')^2 + 1)d^{k-5} - d^{k-3} \equiv -d^{k-7}\lp 1 - 2d^2 + d^4 \rp \equiv 0 \pmod{9}.
    \end{align*}
    Now, if $d^2 \not\equiv 1 \pmod{9}$, i.e. $\{ d, d' \} \equiv \{ 2, 5 \} \textrm{ or } \{ 4, 7 \} \pmod{9}$, we see that $(d')^2 \equiv d^4 \pmod{9}$, and because $d^6 \equiv 1 \pmod{9}$ for any $3\centernot|d$ we obtain
    \begin{align*}
        -d^2 (d')^2 d^{k-7} + (d^2 (d')^2 + 1)d^{k-5} - d^{k-3} \equiv -d^{k-5} \lp d^4 - 2d^2 + 1 \rp \equiv -d^{k-5}(d^2-1)^2 \equiv 0 \pmod{9}
    \end{align*}
    since $d^2 \equiv 1 \pmod{3}$. This completes the proof.
\end{proof}

\subsection{Congruences Modulo 5}
\begin{theorem}
For every $n \ge 0$ and $k \geq 6$ even, the following congruences hold:
\begin{align*}
    b_k(5n+1) \equiv b_k(5n+2) \equiv b_k(5n+3) \equiv 0 \pmod{5}.
\end{align*}
\end{theorem}

\begin{proof}
We begin with the case $k=6$. 
By expanding the respective expressions, we must check that
\begin{align}
    b_6(5n+1) &\equiv \sigma_1(5n+1) - \sigma_3(5n+1) \pmod{5}, \label{pmod eq 5-1} \\
    b_6(5n+2) &\equiv 3\sigma_1(5n+2) - \sigma_3(5n+2) \pmod{5}, \label{pmod eq 5-2} \\
    b_6(5n+3) &\equiv 2\sigma_1(5n+3) - \sigma_3(5n+3) \pmod{5}. \label{pmod eq 5-3}
\end{align}
Since $5 \centernot| d$ for any $d|5n+a$, $a \in \{1,2,3\}$, we may consider for each $d$ its complementary divisor $\frac{5n+a}{d}$.

We begin with $a=1$. If $d \equiv 1, 4 \pmod 5$, then $d^3 \equiv d \pmod 5$. For $d \equiv 2, 3 \pmod 5$, the cofactor simplifies to $\frac{5n+1}{d} \equiv -d \pmod 5$, which implies in particular that $d \neq \frac{5n+1}{d}$. The contribution of this pair to the sum underlying \eqref{pmod eq 5-1} becomes 
\begin{align*}
    d^3 + d + \left(\frac{5n+1}{d}\right)^3 + \left(\frac{5n+1}{d}\right) \equiv d^3 + d - d^3 - d \equiv 0 \pmod 5.
\end{align*}
Summing over all $d \mid (5n+1)$ yields $\sigma_3(5n+1) \equiv \sigma_1(5n+1) \pmod 5$, which fulfills \eqref{pmod eq 5-1}.

Now assume $a \in \{2,3\}$. Since $2,3 \pmod{5}$ are quadratic nonresidues, we note $d \not = \frac{5n+a}{d}$ in this case. By Fermat's Little Theorem, $d^4 \equiv 1 \pmod 5$. Therefore, the cubed cofactors produce
\begin{align*}
d^3 + \left(\frac{5n+a}{d}\right)^3 &\equiv -a\left(d + \frac{5n+a}{d}\right) \pmod 5.
\end{align*}
Summing over all divisors establishes $\sigma_3(5n+a) \equiv -a\sigma_1(5n+a) \pmod 5$, completing the proof of both \eqref{pmod eq 5-2} and \eqref{pmod eq 5-3}.

We now consider $k \geq 8$ even. We must show that
\begin{align}
    b_k(5n+1) &\equiv \sigma_{k-7}(5n+1) + 3\sigma_{k-5}(5n+1) + \sigma_{k-3}(5n+1) \pmod 5, \label{pmod_eq_8-1} \\
    b_k(5n+2) &\equiv 4\sigma_{k-7}(5n+2) + \sigma_{k-3}(5n+2) \pmod 5, \label{pmod eq 8-2} \\
    b_k(5n+3) &\equiv 4\sigma_{k-7}(5n+3) + \sigma_{k-3}(5n+3) \pmod 5. \label{pmod eq 8-3}
\end{align}
 Since $5 \nmid 5n+a$ for $a \in \{1,2,3\}$, every divisor $d \mid 5n+a$ has a unique complimentary divisor $\frac{5n+a}{d}$.
 
Assume that $a=1$. If $k \equiv 0 \pmod 4$, by Fermat's Little Theorem we have $d^{k-7} \equiv d \pmod 5$, $d^{k-5} \equiv d^3 \pmod 5$, and $d^{k-3} \equiv d \pmod 5$. The contribution of the complementary divisor pairs to \eqref{pmod_eq_8-1} then become
\begin{align} \label{pmod 5 helper}
    2(d+d^{-1}) + 3(d^3+d^{-3}) \equiv 0 \pmod 5,
\end{align}
satisfying \eqref{pmod_eq_8-1}. If on the other hand $k \equiv 2 \pmod{4}$, we have $d^{k-7} \equiv d^3 \pmod{5}$, $d^{k-5} \equiv d \pmod{5}$, and $d^{k-3} \equiv d^3 \pmod{5}$, which result in pairs of divisors contributing to \eqref{pmod_eq_8-1} as a multiple of the left-hand side of \eqref{pmod 5 helper}, which completes the proof of \eqref{pmod_eq_8-1}.

Assume that $a\in\{2,3\}$. If $k \equiv 0 \pmod 4$, the contributions of complementary divisor pairs to the linear combinations in \eqref{pmod eq 8-2} and \eqref{pmod eq 8-3} simplify to
\begin{align*}
    4(d + 2d^{-1}) + (d + 2d^{-1}) = 5d + 10d^{-1} \equiv 0 \pmod 5,\\
     4(d + 3d^{-1}) + (d + 3d^{-1}) = 5d + 15d^{-1} \equiv 0 \pmod 5.
\end{align*}
If $k \equiv 2 \pmod 4$, the divisor powers satisfy $d^{k-7} \equiv d^3$ and $d^{k-3} \equiv d^3 \pmod 5$, reducing the pair contributions to 
\begin{align*}
    4(d^3 + 3d^{-3}) + (d^3 + 3d^{-3}) = 5d^3 + 15d^{-3} \equiv 0 \pmod 5, \\ 4(d^3 + 2d^{-3}) + (d^3 + 2d^{-3}) = 5d^3 + 10d^{-3} \equiv 0 \pmod 5.
\end{align*}
Thus, the complementary divisor pairs contribute $0 \pmod 5$ in each subsequent case. Summing over all divisors completes the proof.
\end{proof}

\subsection{Congruences Modulo 7}

\begin{theorem}
For all $n \ge 0$ and $k \geq 6$ even, the following congruences hold:
\begin{align*}
    b_k(7n+3) \equiv b_k(7n+5) \equiv 0 \pmod 7.
\end{align*}
\end{theorem}

\begin{proof}
We begin with $k=6$. For $a \equiv 3, 5 \pmod{7}$, we have $(7n+a)^2 - (7n+a) + 1 \equiv 0 \pmod 7$. Since $6^{-1} \equiv 6 \pmod 7$, we obtain the simplified terms
\begin{align}
    b_6(7n+a) \equiv \sigma_3(7n+a) \pmod 7. \label{pmod eq 7-35}
\end{align}
Because $3$ and $5$ are quadratic non-residues modulo $7$, the equation $d^2 \equiv 3, 5 \pmod 7$ has no solutions, guaranteeing $d \neq \frac{7n+a}{d}$. Note that
\begin{align*}
    \left( \frac{7n+3}{d} \right)^3 \equiv \left(\frac{7n+5}{d}\right)^3 \equiv -d^{-3} \pmod 7.
\end{align*}
The contribution of each complimentary pair of divisors to $\sigma_3(7n+a)$ is then $d^3 - d^{-3} \equiv 0 \pmod 7$ by Fermat's Little Theorem. Summing over all divisors completes the proof. 

We now consider $k \geq 8$ even. In this case, noting that $a^2 \equiv a-1 \pmod{7}$ for $a \in \{ 3,5 \}$, we must show that
\begin{align*}
    b_k(7n+a) \equiv -(a-1)\sigma_{k-7}(7n+a) + a \sigma_{k-5}(7n+a) - \sigma_{k-3}(7n+a) \equiv 0 \pmod{7}.
\end{align*}
If $d|7n+a$, the contribution of this divisor $d$ to the sum above will be precisely
\begin{align*}
    -(a-1)d^{k-7} + ad^{k-5} - d^{k-3} \equiv -d^{k-7}\lp d^4 - ad^2 + (a-1) \rp \pmod{7}.
\end{align*}
In each case, this value vanishes modulo 7 unless $d \equiv \pm (a-1) \pmod{7}$, and the classes $\pm (a-1) \pmod{7}$ correspond to the pairs $\{ 3, 3^{-1} \}, \{ 5, 5^{-1} \} \pmod{7}$, and by summing over complementary divisors we again find vanishing, which completes the proof. 
\end{proof}

\section{Congruences from quadratic non-residues} \label{Sec: Quad Nonresidues}

In this section, we demonstrate the existence of infinitely many non-nested arithmetic progressions for which we have a universal Ramanujan-type congruence for $\Omega$ modulo powers of 2. In this section $\legendre{a}{p}$ denotes the usual Legendre symbol.

\begin{theorem}
    Let $p$ be a prime number such that $p \equiv 1 \pmod 8$. Let $a$ be a quadratic non-residue modulo $p$. Then any integer $n$ such that $m = pn+ a \geq 2$, we have
\begin{align*}
    b_6(pn+a) \equiv 0 \pmod 2.
\end{align*}
\end{theorem}

\begin{proof}
 We must show that
\begin{align}
    b_6(m) \equiv (m^2-m+1)\sigma_1(m)-\sigma_3(m) \equiv 0 \pmod{4} \label{pn+a}
\end{align}
whenever $m$ is a quadratic non-residue for some prime $p \equiv 1 \pmod{8}$.
For any divisor $d|m$, if $d \equiv 2 \pmod 4$, then $d^3 \equiv 0 \pmod 4$. For all $d \not\equiv 2 \pmod 4$, we have $d^3 \equiv d \pmod 4$. Let $m = 2^j M$, where $M$ is odd. We also let $\sigma_0^{\text{odd}}(m)$ denote the number of odd divisors of $m$, so that $\sigma_0^{\text{odd}}(m) = \sigma_0(M)$. Taking the sum over every divisor $ d\mid m$, we have
 \begin{align*}
     \sigma_3(m) \equiv \sigma_1(m) - 2\sigma_0(M) \pmod 4.
 \end{align*}
 Further substitution into \eqref{pn+a} and through the reduction of polynomial coefficients modulo $4$, 
 \begin{align}
    b_6(m) 
    &\equiv (m^2-m+1) \sigma_1(m)- (\sigma_1(m) -2\sigma_0(M)) \pmod 4 \nonumber \\
    &\equiv m(m+1)\sigma_1(m)+2\sigma_0(M) \pmod 4. \label{reduced pn+a}
\end{align}

For $m \equiv 1 \pmod 4$, we have $j = 0$ and $m=M=4k+1$. The leading coefficient satisfies $m(m+1)=(4k+1)(4k+2) \equiv 2 \pmod 4$. Since $M$ is odd, every divisor $d \mid M$ can be paired with another odd divisor $M/d$. This implies $\sigma_1(M) \equiv \sigma_0(M) \pmod 2$, yielding $2\sigma_1(M) \equiv 2\sigma_0(M) \pmod 4$. Substitution back into \eqref{reduced pn+a} produces
 \begin{align*}
    b_6(m) \equiv 2\sigma_1(M)+2\sigma_0(M) \equiv 4\sigma_0(M) \equiv 0\pmod 4.
 \end{align*}

 For $m \equiv 2\pmod 4$, $j=1$ and $m=2M$, where $M=2k+1$. The algebraic term yields $m(m+1)= 2M(2M+1) \equiv 2(1)(3) \equiv 2 \pmod 4$. By the multiplicativity of $\sigma_1$,
 \begin{align*}
    \sigma_1(m) = \sigma_1(2)\sigma_1(M)=3\sigma_1(M) \equiv -\sigma_1(M) \pmod 4.
 \end{align*}
 Evaluating the given in the context of \eqref{reduced pn+a},
 \begin{align*}
    b_6(m) \equiv -2\sigma_1(M)+2\sigma_0(M) \pmod 4.
 \end{align*}
 Applying $2\sigma_1(M) \equiv 2\sigma_0(M) \pmod 4$, 
 \begin{align*}
    b_6(m) \equiv -2\sigma_0(M)+2\sigma_0(M) \equiv 0 \pmod 4.
 \end{align*}

 For $m \equiv 3 \pmod 4$, $j=0$ and $m=M=4k+3$. Its product, $m+1=4k+4 \equiv 0 \pmod 4$, eliminates the first term entirely, producing 
 \begin{align*}
        b_6(m) \equiv 2\sigma_0(M) \pmod 4.
 \end{align*}
 Further, let $m \equiv a\pmod p$, since $a$ is a quadratic non-residue and $m = M$, $\sigma_0(m) = \sigma_0(M)$ is even. Thus we have
 \begin{align*}
    b_6(m) \equiv 2\sigma_0(M) \equiv 0 \pmod 4,
 \end{align*}
 which completes the proof in this case.

 For $m \equiv 0 \pmod 4$, $j \geq 2$, implying that $m(m+1) \equiv 0 \pmod 4$. As before, we have $b_6(m) \equiv 2\sigma_0(M) \pmod 4$.
 If $j$ is even, then since $m$ is a quadratic nonresidue modulo $p$, we have the equation of Legendre symbols $-1 = \legendre{m}{p} = \legendre{2}{p}^j \legendre{M}{p} = \legendre{M}{p}$, so in particular $M$ is a quadratic nonresidue modulo $p$.
 If $j$ is odd, then we have
 \begin{align*}
    m=2(2^{\frac{j-1}{2}}y)^2.
 \end{align*}
 As $p \equiv 1 \pmod 8$, the law of quadratic reciprocity dictates that $\legendre{2}{p} =1$, and this again implies $\legendre{M}{p} = -1$. Therefore, in any case we must have that $M$ is a quadratic nonresidue modulo $p$. In particular, $M$ is not a square, so $\sigma_0(M)$ is even and it follows again that
\begin{align*}
    b_6(m) \equiv 2\sigma_0(M) \equiv 0 \pmod 4,
\end{align*}
which completes the proof.
\end{proof}

Although the theorem above does not extend to all $H_k$, a slight modification does, which we prove below.

\begin{theorem} \label{Thm: even progressions}
    Let $p$ be an odd prime, and let $a$ be an odd integer such that $\legendre{a}{p} = -1$. For any even integer $k \geq 6$,
    \begin{align*}
        b_k(2pn+a) \equiv 0 \pmod{16}.
    \end{align*}
\end{theorem}

\begin{proof}
    We first prove that the condition that an integer $m \geq 1$ belongs to some arithmetic progression $2pn+a$ for at least one prime $p$ and quadratic nonresidue $a \pmod{2p}$ with $a$ odd if and only if $m$ is not a perfect square. Since perfect squares are residues in any modulus, we need only consider the converse. So, suppose that $m \geq 1$ is not a perfect square. We may assume without loss of generality that $m$ is squarefree since $\legendre{k^2}{p}=1$ for all $k$ and $p \nmid k$, and we may assume $m$ is odd since there are infinitely many primes $p$ with $\legendre{2}{p}$ by Dirichlet's theorem and quadratic reciprocity. We then write $m = q_1 \cdots q_r$ as a product of odd primes. Pick an integer $b$ so that $\legendre{b}{q_1} = -1$, and we construct a residue class $x$ by the Chinese Remainder Theorem that satisfies $x \equiv 1 \pmod{4}$, $x \equiv b \pmod{q_1}$, and $x \equiv 1 \pmod{q_j}$ for $j = 2, 3, \dots, r$. By Dirichlet's theorem, there is a prime $p \equiv x \pmod{4m}$, and this prime satisfies $p \equiv 1 \pmod{4}$. Using the law of quadratic reciprocity and the multiplicativity of the Legendre symbol, we then have
    \begin{align*}
        \legendre{m}{p} = \prod_{j=1}^r \legendre{q_j}{p} = \prod_{j=1}^r \legendre{p}{q_j} = -1,
    \end{align*}
    and so $m$ is both odd and a quadratic nonresidue modulo $p$, and thus $m = 2pn+a$ for some $a$ which is a quadratic nonresidue modulo $p$.
    
    We now begin the proof proper. By the prior discussion, it will suffice to prove that for every $m$ which is not a perfect square, we have
    \begin{align*}
        (m^2-m+1)\sigma_1(m)-\sigma_3(m) &\equiv 0 \pmod {32}
    \end{align*}
    and for every $k \geq 8$ even
    \begin{align*}
        -m^2 \sigma_{k-7}(m) + (m^2+1)\sigma_{k-5}(m) - \sigma_{k-3}(m) \equiv 0 \pmod{128}.
    \end{align*}
    Since $m$ is not a square, every divisor $d \mid m$ is not equal to its complementary divisor $\frac{m}{d}$. For $k = 6$, the contribution of a pair $\{d, m/d\}$ to $(m^2-m+1)\sigma_1(m)-\sigma_3(m)$ is given by 
    \begin{align*}
        (m^2-m+1)\left(d + \frac{m}{d}\right) - \left(d^3 + \frac{m^3}{d^3}\right) &= \left(d + \frac{m}{d}\right) \left[ (m^2-m+1) - \left(d^2 - m + \frac{m^2}{d^2}\right) \right] \\
        &= \left(d + \frac{m}{d}\right) \left(m^2 - d^2 - \frac{m^2}{d^2} + 1\right) \\
        &= \left(d + \frac{m}{d}\right) (d^2 - 1) \left(\frac{m^2}{d^2} - 1\right).
    \end{align*}

    Since $d$ and $m/d$ are odd integers, we have $(d^2 - 1)\left(\frac{m^2}{d^2} - 1\right) \equiv 0 \pmod{64}$ since the square of an odd number is always $1 \pmod{8}$, and this establishes the $k=6$ case.

    For any even integer $k \ge 8$, the contribution of $\{d, m/d\}$ to $-m^2 \sigma_{k-7}(m) + (m^2+1)\sigma_{k-5}(m) - \sigma_{k-3}(m)$ is given by 
    \begin{align*}
        &-m^2\left(d^{k-7} + \left(\frac{m}{d}\right)^{k-7}\right) + (m^2+1)\left(d^{k-5} + \left(\frac{m}{d}\right)^{k-5}\right) - \left(d^{k-3} + \left(\frac{m}{d}\right)^{k-3}\right) \\
        &\quad = (d^2 - 1)\left(\frac{m^2}{d^2} - 1\right) \left[ d^{k-5} + \left(\frac{m}{d}\right)^{k-5} \right].
    \end{align*}

    As before, $(d^2 - 1)\left(\frac{m^2}{d^2} - 1\right) \equiv 0 \pmod{64}$. Furthermore, for any even $k \ge 8$, the $d^{k-5} + (m/d)^{k-5}$ is the sum of two odd integers and is therefore even. Therefore, each summand is divisible by $128$ and the proof is complete.
\end{proof}

\section{Proof of Theorem \ref{Thm: Infinitely many uniform congruences}} \label{Sec: Serre}

The proof of Theorem \ref{Thm: Infinitely many uniform congruences} follows from the analogous non-uniform version of this result for mixed weight quasimodular forms, as stated below.

\begin{theorem} \label{Thm: Infinitely many congruences}
    Let $f \in \widetilde{M}_{\leq k} \cap \Z[[q]]$ have $\sum_{n \geq 0} a_f(n) q^n$, and let $M \gg 0$ be an integer. Then there are infinitely many non-nested arithmetic progressions $An + B$ such that
    \begin{align*}
        a_f\lp A n + B \rp \equiv 0 \pmod{M}.
    \end{align*}
\end{theorem}

\begin{proof}[Sketch of Proof]
    Since $E_2$ is a modular form modulo $M$ due to Serre \cite{Serre p-adic}, we need only consider mixed weight modular forms, and since a mixed weight modular form contains only finitely many fixed weight components, it suffices to consider modular forms of pure weight modulo $M$. The result is a much simplified variation of Ono's argument for the analogous result for the partition function \cite{Ono}, namely, it leverages results of Serre on the Chebotarev density theorem, Hecke operators and Galois representations \cite{SerreDiv,SerreChebotarev}. Precise details of the argument in this case are exactly analogous to the more recent work \cite[Theorem 1.7]{AmdeberhanOnoSingh}, which also deals with mixed weight quasimodular forms.
\end{proof}

We may then prove the rest of Theorem \ref{Thm: Infinitely many uniform congruences}.

\begin{proof}[Proof of Theorem \ref{Thm: Infinitely many uniform congruences}]
    Every form $H_k$ individually satisfies the criteria of Theorem \ref{Thm: Infinitely many congruences}, and so each $H_k$ has infinitely many non-nested arithmetic progressions $An+B$ for which $b_k(An+B) \equiv 0 \pmod{M}$. Since by Euler's theorem the $H_k$ only occupy only finitely many distinct classes in $\lp \Z/M\Z \rp[[q]]$ for any $M \geq 2$ (since $H_k \equiv H_{k+\phi(M)} \pmod{M}$ when $k \gg 0$ is large enough so that $d^k$ is either coprime to $M$ or is forced to have prime factors of multiplicity larger than those of $M$), one can construct infinitely many arithmetic progressions for which the congruence $b_k(An+B) \equiv 0 \pmod{M}$ holds uniformly in $k$ by the Chinese Remainder Theorem. This completes the proof.
\end{proof}

\section{Proof of Theorem \ref{Thm: Multiplicative Congruences}} \label{Sec: Multiplicative}

In this section, we utilize the multiplicative structure of divisor sums and elementary divisibility lemmas to prove Theorem \ref{Thm: Multiplicative Congruences}. To this end, in this section we will denote by $\nu_p(n)$ the multiplicity of the prime $p$ as a factor of $n$. We begin with the following lemma.

\begin{lemma} \label{Lem: Divisibility}
    Suppose $p>2$ is a prime, and let $x,y$ be integers with $x \equiv y \not \equiv 0 \pmod{p}$. Then for any integer $n \geq 1$, we have
    \begin{align*}
        \nu_p\lp x^n - y^n \rp = \nu_p\lp x-y \rp + \nu_p(n).
    \end{align*}
\end{lemma}

\begin{proof}
    The result follows since
    \begin{align*}
        x^n - y^n = \lp x-y \rp \lp x^{n-1} + x^{n-2} y + x^{n-3} y^2 + \dots + y^n \rp
    \end{align*}
    and
    \begin{align*}
        x^{n-1} + x^{n-2} y + x^{n-3} y^2 + \dots + y^n \equiv n x^{n-1} \pmod{p},
    \end{align*}
    so $\nu_p\lp x^{n-1} + x^{n-2} y + x^{n-3} y^2 + \dots + y^n \rp = \nu_p(n)$.
\end{proof}

From this lemma, we are able to deduce the following divisibility result for $\sigma_k$-values for powers of primes.

\begin{proposition} \label{Prop: Divisors}
    Let $p>2$ be prime, and let $k \geq 1$ be an odd integer, and let $d = \ord_p(2)$, i.e. $2^d-1 \equiv 0 \pmod{p}$ with minimal $d \geq 1$. Then for any $\ell \geq 1$, we have
    \begin{align*}
        \nu_p\lp \sigma_k(2^\ell) \rp = \begin{cases}
            0 & d \nmid k(\ell+1), \\
            \nu_p(\ell+1) & d|k, \\
             \nu_p(2^d-1) + \nu_p(\ell+1) + \nu_p(k) - \nu_p(d) & \text{otherwise}.
        \end{cases}
    \end{align*}
\end{proposition}

\begin{proof}
    We can begin by expressing the divisor sum as the geometric series 
    \begin{align*}
        \sigma_k(2^{\ell}) = \sum_{j=0}^{\ell} (2^k)^j = \frac{(2^k)^{\ell+1} - 1}{2^k - 1}.
    \end{align*}
    If $d\nmid k(\ell+1)$, then it is immediately clear that $\nu_p\lp \sigma_k(2^\ell) \rp = 0$.
    
    So assume that $d\mid k(\ell+1)$. Applying $\nu_{p}$ to the geometric formula for $\sigma_k(2^\ell)$ yields
    \begin{align*}
         \nu_p\lp \sigma_k(2^{\ell}) \rp = \nu_p\lp (2^k)^{\ell+1} - 1 \rp - \nu_p(2^k - 1).
    \end{align*}
    If $d|k$, then by Lemma \ref{Lem: Divisibility}, we have
    \begin{align*}
        \nu_p\lp 2^{k(\ell+1)}-1^{\ell+1} \rp = \nu_p(2^k-1) + \nu_p(\ell+1).
    \end{align*}
    Therefore, when $d|k$ we have
    \begin{align*}
        \nu_p(\sigma_k(2^\ell)) = \lp \nu_p(2^k-1) + \nu_p(\ell+1) \rp - \nu_p(2^k-1) = \nu_p(\ell+1).
    \end{align*}
    If $d \nmid k$, then $\nu_p(2^k-1)=0$ and Lemma \ref{Lem: Divisibility} implies
    \begin{align*}
        \nu_p\lp \sigma_k(2^\ell) \rp = \nu_p\lp (2^d)^{\frac{k(\ell+1)}{d}} - 1^{\frac{k(\ell+1)}{d}} \rp &= \nu_p(2^d-1) + \nu_p\lp \frac{k(\ell+1)}{d} \rp \\
        &= \nu_p(2^d-1) + \nu_p(\ell+1) + \nu_p(k) - \nu_p(d).
    \end{align*}
    This completes the proof.
\end{proof} 

\begin{theorem} \label{Thm: Multiplicative Congruences}
    Let $p>3$ be a prime, and let $d = \ord_p(2)$. Then for every even $k \geq 6$ and $n \geq 0$, we have
    \begin{align*}
        b_k\lp 2^{dp}n + 2^{dp-1} \rp \equiv 0 \pmod{p}.
    \end{align*}
\end{theorem}

\begin{remark}
    Note that Corollary \ref{Cor: Mersenne primes} follows from Theorem \ref{Thm: Multiplicative Congruences} by setting $2^d-1=p$.
\end{remark}

\begin{proof}
    Let $d = \ord_p(2)$, and let $k \geq 6$ be any even integer. Consider the values
    $\sigma_{k-1}\lp 2^{dp}n + 2^{dp-1} \rp$ for $n \geq 0$. Then by the definition of $b_k(n)$, if the values $\sigma_{k-1}\lp 2^{dp-1} \rp$ have a common prime divisor $p>3$ for every $k \geq 6$ even, then the desired universal congruence follows by the multiplicativity of $\sigma_{k-1}(n)$. Now, by Proposition \ref{Prop: Divisors} we have for any prime $p>3$ that
    \begin{align*}
        \nu_p\lp \sigma_k(2^{dp-1}) \rp = \begin{cases}
            \nu_p(dp) & d|(k-1), \\
             \nu_p(2^d-1) + \nu_p(dp) + \nu_p(k-1) - \nu_p(d) & \text{otherwise}.
        \end{cases}
    \end{align*}
    Since $\nu_p(dp) = \nu_p(d) + 1 > \nu_p(d)$, each case above yields a positive contribution, and therefore $p|\sigma_{k-1}(2^{dp-1})$ for every $k$ required, and so the congruence follows by multiplicativity of $\sigma_{k-1}$.
\end{proof}

\section{Conclusion and Conjectures} \label{Sec: Conclusion}

Due to Theorem \ref{Thm: Infinitely many uniform congruences}, we know that many additional moduli remain unexplored. Because of the way the proof proceeds, with many embedded applications of the Chinese Remainder Theorem and calculations in large Galois groups implicit in Serre's work, these congruences are difficult to identify explicitly. It is certainly worthwhile to search out and prove more of these congruences; we propose in this final section a number of directions which are natural to pursue based upon our data.

\subsection{Conjectural classification of $\Omega$ by congruences}

Results like Theorem \ref{Thm: Infinitely many uniform congruences} suggest the potential for very strong applications in modular form theory. The existence (and explicitly construction of) these uniformly valid congruences creates testing criteria for whether a randomly chosen mixed weight quasimodular for $f \in \widetilde{M}_{\leq k}$ could be a prime-detecting form, since any $f$ not satisfying even a single such congruence cannot be in the space $\Omega$ generated by the prime-detecting forms. The collection of all $f \in \widetilde{M}_{\leq k} \cap \Z[[q]]$ whose coefficients $a_f(n)$ satisfy a particular congruence $a_f(An+B) \equiv 0 \pmod{M}$ is the kernel of a homomorphism $\widetilde{M}_{\leq k} \longrightarrow \lp \Z/M\Z \rp [[q]]$ defined by
\begin{align*}
    \sum a_f(n) q^n \mapsto \sum a_f(An+B) q^n \pmod{M}.
\end{align*}
Since the target space of such a homomorphism is a space of mixed weight quasimodular forms with higher level and with Nebentypus (due to classical trick of sieving via Dirichlet characters), the target space is finite-dimensional and so congruences are verifiable with a finite computation with an explicit basis, despite the fact that Sturm bounds are not known yet for such spaces (see \cite{CraigSturm} for partial progress in this direction).

In principle, if kernels of such maps are sufficiently well understood by a more developed theory of mixed weight modular and quasimodular forms, then it would be natural to study the space $\Omega$ as a subspace of these kernels. In particular, we propose the following local-to-global type question.

\begin{question}
    Given $k \geq 6$ even and $f \in \widetilde{M}_{\leq k} \cap \Z[[q]]$, is there an explicit (infinite) list of tuples $(A_j, B_j, M_j)$ such that $f$ obeys the Ramanujan-type congruences $a_f(A_j n + B_j) \equiv 0 \pmod{M_j}$ if and only if $f \in \Omega$?
\end{question}

No finite collection of moduli could resolve such a question, but Theorem \ref{Thm: Infinitely many uniform congruences} leaves open whether such a question can be answered affirmatively using some local-to-global argument. It is plausible that the results already proven in this work can only be satisfied by an element of the so-called Eisenstein space $\mathcal E$ of \cite{CraigIttersumOno}, generated linearly by Eisenstein series and their derivatives (and thus ignoring purely cuspidal components). It would be interesting to prove this rigorously.

\subsection{Analogs to Ramanujan's original work}

The study of congruences of the type we have mentioned originated in Ramanujan's study of $p(n)$, in which he wrote as follows \cite{Ramanujan}: \\[+0.1cm]

{\it I have proved a number of arithmetical properties of $p(n)$, in particular that
\begin{align*}
    p(5n+4) \equiv 0 \pmod{5}
\end{align*}
and
\begin{align*}
    p(7n+5) \equiv 0 \pmod{7} \dots
\end{align*}
I have since found another method which enables me to prove all of these properties and a variety of others, of which the most striking is
\begin{align*}
    p(11n+6) \equiv 0 \pmod{11}.
\end{align*}
There are corresponding properties in which the moduli are powers of 5, 7, or 11 $\dots$ It appears that there are no equally simple properties for any moduli involving primes other than these three.} \\[+0.1cm]

There are two key aspects of Ramanujan's statements which we have not addressed in the context of prime-detecting quasimodular forms. The first, the presence of congruences modulo powers of 5, 7, 11, has as its essence the idea that there should be a natural sequence of nested subprogressions of $5n+4$ which vanish modulo $5^m$ for each $m \geq 1$, and likewise for the primes 7 and 11. In the case of the partition function, these take the particular form
\begin{align*}
    p\left(\ell^{\alpha}n + \delta_{\ell}\right) \equiv 0 \pmod{\ell^{\alpha}}, \ \ \ \ \ \textrm{for } \ell = 5, 11, \ \ \ 24 \delta_\ell \equiv -1 \pmod{\ell^\alpha}
\end{align*}
and
\begin{align*}
    p\left(7^{\alpha}n + \delta_{7}\right) \equiv 0 \pmod{7^{\lfloor \alpha/2 \rfloor + 1}}, \ \ \ \ \ \textrm{for } 24\delta_7 \equiv -1 \pmod{7^\alpha}.
\end{align*}
Based on data gathered during the course of our research, we project that such results will hold for the primes $2, 3, 5, 7$ for which there were congruences of this particular type.

We also note Ramanujan's second comment, that for $p(n)$ the prime $11$ appeared to be the end of this observed pattern. The interpretation of Ramanujan's comment is that his three congruences are the only ones of the type
\begin{align*}
    p\lp \ell n + \beta \rp \equiv 0 \pmod{\ell}
\end{align*}
with $\ell$ prime; this was proven by Ahlgren and Boylan \cite{AhlgrenBoylan} using the filtration theory of modular forms modulo $p$. Based on our data, we project that the primes so far identified as also the only examples for which such congruences exist not just in the universal sense, but even for any particular $H_k$.

\begin{conjecture}
    For $k \geq 6$ even, $\ell$ prime and $0 \leq \beta < \ell$, any congruence
    \begin{align*}
        b_k\lp \ell n + \beta \rp \equiv 0 \pmod{\ell}
    \end{align*}
    must have $\ell = 2, 3, 5, 7$.
\end{conjecture}

This does not mean, however, that no prime-detecting quasimodular forms can have congruences of this type. In particular, consider the prime-detecting forms defined for odd integers $\ell > k \geq 1$ by
\begin{align*}
    f_{k,\ell}(q) := (D^\ell+1)G_{k+1} - (D^k+1)G_{\ell+1}.
\end{align*}
The Fourier coefficients of such forms are given by
\begin{align*}
    c_{k,\ell}(n) := \lp n^\ell+1 \rp \sigma_k(n) - \lp n^k + 1 \rp \sigma_\ell(n).
\end{align*}
Since $k,\ell$ are odd, $n+1$ divides each of $n^\ell+1$ and $n^k+1$, from which it follows that for every prime $p$ and every $n \geq 1$, we have
\begin{align*}
    c_{k,\ell}\lp pn - 1 \rp \equiv 0 \pmod{p}.
\end{align*}
It follows that the particular forms $f_{k,\ell}$ have these Ramanujan-type congruences for every prime. It is also straightforward to see that if $f \in \Omega$, then the $n$th coefficient of $Df$ vanishes modulo $p$ whenever $p|n$. It would be interesting to study further the presence of Ramanujan-type congruences for general prime-detecting quasimodular forms, to determine whether there are nontrivial examples of this phenomenon beyond those already discovered. We pose the following question as potentially interesting for further study in this area.

\begin{question}
    Which prime-detecting quasimodular forms $f$ possess a congruence $a_f(\ell n + \beta) \equiv 0 \pmod{\ell}$ for all primes $\ell$?
\end{question}

\subsection{Additional remarks}

In this paper, we have been able to enumerate a variety of congruences of different types for the prime-detecting quasimodular forms $H_k$. These congruences are by no means exhaustively studied, as our focus has been on enumerating those congruences which are shared in a uniform manner across all $H_k$. In many instances, certain $H_k$ have higher divisibilities than others in the progressions already studied than others do, and some $H_k$ have entirely separate progressions possessing divisibilities that other $H_k$ do not have at all.

We note that the congruences proven in this work mirror in many ways those discovered by Motomura and Suda \cite{MotomuraSuda} for certain MacMahon-type functions. Since these are related to the $H_k$ by the work of \cite{CraigIttersumOno}, it would be of interest to explore whether these types of congruences hold for general families of $q$-multiple zeta values, from which both quasimodular forms and the MacMahon-type functions emerge as special cases. It is probable that higher level $q$-multiple zeta values and prime-detecting quasimodular forms (see e.g. \cite{Craig}) have similar congruence results. It would also be interesting to determine whether the methods of \cite{CraigIttersumOno} through which $H_k$ are represented by $q$-multiple zeta values shed light either on additional congruences for $H_k$ or lead to new congruences discovered for coefficients of $q$-multiple zeta values.


\begin{thebibliography}{99}

\bibitem{AhlgrenBoylan} S. Ahlgren, M. Boylan, \emph{Arithmetic properties of the partition function}. Invent. Math. {\bf 153}, no. 3, (2003) 487--502.

\bibitem{AmdeberhanOnoSingh} T. Amdeberhan, K. Ono, A. Singh. \emph{MacMahon's sums-of-divisors and allied q-series}.
Adv. Math. {\bf 452} (2024), Paper No. 109820, 25 pp.

\bibitem{Craig} W. Craig. \emph{Higher level $q$-multiple zeta values with applications to quasimodular forms and partitions}. Mathematische Annalen {\bf 396}, 2 (2026).

\bibitem{CraigSturm} W. Craig. \emph{New types of Sturm bounds via $p$-adic transfer methods}. Preprint, \url{https://arxiv.org/abs/2602.10240}.

\bibitem{CraigIttersumOno} W. Craig, J-W. van Ittersum, K. Ono. \emph{Integer partition detect the primes}. Proc. Nat. Ac. Sci., {\bf 121} (39) (2024), e2409417121.

\bibitem{CraigMerca} W. Craig, M. Merca. \emph{On Ramanujan-type congruences for multiplicative functions}. Revista de la Real Academia de Ciencias Exactas, F\'{i}sicas y Naturales. Serie A. Matem\'{a}ticas {\bf 116}, 128 (2022).

\bibitem{KaneKrishnamoorthyLau} B. Kane, K. Krishnamoorthy, Y-K. Lau, \emph{On a conjecture about prime-detecting quasimodular forms}. Res. Math. Sci. {\bf 12}, 60 (2025).

\bibitem{KanekoZagier} M. Kaneko, D. Zagier. \emph{A generalized Jacobi theta function and quasimodular forms}.
In: The Moduli Spaces of Curves (R. Dijkgraaf, C. Faber, G. v.d. Geer, eds.), Prog. in Math. {\bf 129}, Birkhäuser, Boston (1995) 165--172.

\bibitem{MotomuraSuda} Y. Motomura, T. Suda, \emph{Congruences for the coefficients of MacMahon-like $q$-series}. Preprint, \url{https://arxiv.org/abs/2607.22349}.

\bibitem{Ono} K. Ono, \emph{Distribution of the partition function modulo $m$}.
Ann. of Math. (2) {\bf 151} (2000), no. 1, 293--307.

\bibitem{Ramanujan} S. Ramanujan, \emph{Congruence properties of partitions}. Proc. London Math. Soc. {\bf 18} (1920), xix.

\bibitem{SerreDiv} J-P. Serre, \emph{Divisibilité de certaines fonctions arithmétiques}. L'Enseignement Mathématique, {\bf 22} (3--4), 227--260.

\bibitem{Serre p-adic} J-P. Serre, \emph{Formes modulaires et fonctions z$\hat{e}$ta $p$-adiques}. Modular functions of one variable, III (Proc. Internat. Summer School, Univ. Antwerp, 1972), 1973, pp. 191--268. Lecture
Notes in Math., Vol. 350.

\bibitem{SerreChebotarev} J-P. Serre, \emph{Quelques applications du théorème de densité de Chebotarev}. Inst. Hautes Études Sci. Publ. Math. No. {\bf 54} (1981), 123--201.

\end{thebibliography}
\end{document}